\documentclass[a4paper]{amsart}

\usepackage{amsfonts}
\usepackage{amssymb}

\usepackage[T1]{fontenc}
\usepackage[utf8]{inputenc}
\usepackage{enumerate}

\usepackage[abbrev]{amsrefs}

\usepackage{backref}

\newcommand{\abs}[1]{\mathopen\lvert#1\mathclose\rvert}
\newcommand{\bigabs}[1]{\bigl\lvert#1\bigr\rvert}

\newcommand{\biggabs}[1]{\biggl\lvert#1\biggr\rvert}
\newcommand{\norm}[1]{\mathopen\lVert#1\mathclose\rVert}
\newcommand{\bignorm}[1]{\mathopen\big\lVert#1\mathclose\big\rVert}
\newcommand{\floor}[1]{\lfloor#1\rfloor}
\newcommand{\N}{{\mathbb N}}
\newcommand{\R}{{\mathbb R}}

\renewcommand{\S}{{\mathbb S}}

\DeclareMathOperator{\Lip}{Lip}
\DeclareMathOperator{\dist}{dist}

\DeclareMathOperator{\Int}{int}

\newcommand{\dif}{\,\mathrm{d}}

\theoremstyle{plain}
\newtheorem{proposition}{Proposition}[section]
\newtheorem{lemma}[proposition]{Lemma}
\newtheorem{theorem}{Theorem}

\newtheoremstyle{addendumstyle}{\topsep}{\topsep}{\itshape}{}{\bfseries}{.}{.5em plus 1pt minus 1pt}{#1 #2 to #3}
\theoremstyle{addendumstyle}

\theoremstyle{definition}

\theoremstyle{remark}

\newtheorem*{Claim}{Claim}

\usepackage{refcount}
\newcounter{cte}
\newcommand{\Constant}{\refstepcounter{cte} C_{\thecte}}
\newcommand{\NewConstant}{\setcounter{cte}{1} C_{\thecte}}
\newcommand{\SameConstant}{C_{\thecte}}

\numberwithin{equation}{section}

\title{Strong approximation of fractional Sobolev maps}

\author{Pierre Bousquet}
\address{
Aix-Marseille Universit\'e\\
Laboratoire d'analyse, topologie, probabilit\'es UMR7353\\
39 rue Fr\'ed\'eric Joliot Curie\\
13453 Marseille Cedex 13\\
France}

\email{Pierre.Bousquet@univ-amu.fr}

\author{Augusto C. Ponce}
\address{
 Université catholique de Louvain\\
 Institut de Recherche en Math{\'e}matique et Physique\\
 Chemin du cyclotron 2, bte L7.01.02\\
1348 Louvain-la-Neuve\\
Belgium}
\email{Augusto.Ponce@uclouvain.be}

\author{Jean Van Schaftingen}
\address{
 Université catholique de Louvain\\
 Institut de Recherche en Math{\'e}matique et Physique\\
 Chemin du cyclotron 2, bte L7.01.02\\
1348 Louvain-la-Neuve\\
Belgium}
\email{Jean.VanSchaftingen@uclouvain.be}

\begin{document}

\begin{abstract}
Brezis and Mironescu have announced several years ago that for a compact manifold \(N^n \subset \R^\nu\) and for real numbers \(0 < s < 1\) and \(1 \le p < \infty\), the class \(C^\infty(\overline{Q}^m; N^n)\) of smooth maps on the cube with values into \(N^n\) is dense with respect to the strong topology in the Sobolev space \(W^{s, p}(Q^m; N^n)\) when the homotopy group \(\pi_{\lfloor sp \rfloor}(N^n)\) of order \(\lfloor sp \rfloor\) is trivial.
The proof of this beautiful result is long and rather involved.
Under the additional assumption that \(N^n\) is \(\floor{sp}\) simply connected, we give a shorter proof of their result. 
Our proof for \(sp \ge 1\) is based on the existence of a retraction of \(\R^\nu\) onto \(N^n\) except for a small subset in the complement of \(N^n\) and on the Gagliardo-Nirenberg interpolation inequality for maps in \(W^{1, q} \cap L^\infty\).
In contrast, the case \(sp < 1\) relies on the density of step functions on cubes in \(W^{s, p}\).
\end{abstract}



\maketitle

\section{Introduction}

We address in this paper the problem of density of smooth maps in the fractional Sobolev spaces \(W^{s, p}\) with values into manifolds.
More precisely, let \(0 < s < 1\) and \(1 \le p < +\infty\), and let \(N^n\) be a compact manifold of dimension \(n\) imbedded in the Euclidean space \(\R^\nu\).
The class of Sobolev maps \(W^{s, p}(Q^m; N^n)\) on the unit \(m\) dimensional cube \(Q^m\) with values into \(N^n\) is defined as the set of measurable maps \(u : Q^m \to \R^\nu\) such that
\[
u(x) \in N^n \quad \text{for a.e.\@ \(x \in Q^m\)}
\]
having finite Gagliardo seminorm~\cite{Gagliardo},
\[
[u]_{W^{s, p}(Q^m)} 
= \bigg( \int\limits_{Q^m}\int\limits_{Q^m} \frac{|u(x)- u(y)|^p}{|x-y|^{m+s p}} \dif x \dif y \bigg)^{1/p}.
\]
The following question arises naturally: does \(W^{s, p}(Q^m; N^n)\) coincide with the closure of smooth maps $C^{\infty}(\overline Q^m; N^n)$ with respect to the distance given by
\[
d_{s, p}(u, v)
= \norm{u - v}_{L^p(Q^m)}  + [u - v]_{W^{s, p}(Q^m)}?
\]

This is indeed the case when \(sp \ge m\):

\begin{proposition} \label{premiertheorem}
If $sp\geq m$, then the family of smooth maps $C^{\infty}(\overline Q^m; N^n)$ is strongly dense in $W^{s,p}(Q^m; N^n)$.
\end{proposition}

Here is the sketch of the argument:
given \(u \in W^{s, p}(Q^m; N^n)\), we consider the convolution \(\varphi_\epsilon \ast u\) with a smooth kernel \(\varphi_\epsilon\). 
If the range of \(\varphi_\epsilon \ast u\) lies in a small tubular neighborhood of \(N^n\), then we may project \(\varphi_\epsilon \ast u\) pointwisely into \(N^n\). 
We can always do this for \(\epsilon > 0\) sufficiently small as long as \(sp \ge m\).
Indeed, in this case \(W^{s, p} (Q^m;\R^\nu)\) imbeds into the space of functions of vanishing mean oscillation \(\mathrm{VMO} (Q^m;\R^\nu)\), whence \(\dist{(\varphi_\epsilon \ast u, N^n)}\) converges uniformly to \(0\) \cite{Brezis-Nirenberg_1995}*{Eq.~(7)}.

The counterpart of Proposition~\ref{premiertheorem} for \(W^{1, p}(Q^m; N^n)\) and \(p \ge m\) is due to Schoen and Uhlenbeck~\cite{Schoen-Uhlenbeck}.
The role played by \(\mathrm{VMO}\) functions in this problem has been first observed by Brezis and Nirenberg~\cite{Brezis-Nirenberg_1995}.

In the subtler case \(sp < m\), the answer to the density problem only depends on the topology of the manifold \(N^n\):

\begin{theorem}\label{theoremBrezisMironescu}
If \(sp < m\), then \(C^\infty(\overline Q^m; N^n)\) is strongly dense in \(W^{s, p}(Q^m; N^n)\) if and only if \(\pi_{\floor{sp}}(N^n) \simeq \{0\}\).
\end{theorem}

We denote by \(\floor{sp}\) the integral part of \(sp\) and for every \(\ell \in \N\), \(\pi_{\ell}(N^n)\) is the \(\ell\)th homotopy group of \(N^n\). 
The topological assumption \(\pi_{\floor{sp}}(N^n) \simeq \{0\}\) means that every continuous map \(f : \S^{\floor{sp}} \to N^n\) on the \(\floor{sp}\) dimensional sphere is homotopic to a constant map.
The necessity of this condition has been known for some time \citelist{\cite{Escobedo}*{Theorem~3} \cite{Schoen-Uhlenbeck}*{Section~4, Example}\cite{Mironescu}*{Theorem~4.4}}.

Brezis and Mironescu  have announced this beautiful result in a personal communication  in April 2003 and a sketch of the proof can be found for instance in \cite{Mironescu_2004}*{pp.~205--206}.
The analog of Theorem~\ref{theoremBrezisMironescu} for \(W^{1, p}\) Sobolev maps had been obtained by Bethuel in his seminal paper~\cite{Bethuel} (see also \cite{Hang-Lin}).
Partial results for fractional Sobolev exponents \(s\) were known when the manifold \(N^n\) is a sphere with dimension \(n \ge sp\) \cite{Escobedo} and also in the setting of trace spaces with \(s = 1 - \frac{1}{p}\) \citelist{\cite{Bethuel-1995} \cite{Mucci}}.

The proof of Theorem~\ref{theoremBrezisMironescu} is long and quite involved.
In this paper we prove the reverse implication of Theorem~\ref{theoremBrezisMironescu} in the case of  \(\floor{sp}\) simply connected manifolds \(N^n\).
Under this assumption, we give a shorter argument which leads to the following:

\begin{theorem}\label{deuxiemetheorem}
If \(sp < m\) and if  for every \(\ell \in \{0, \dots, \floor{sp}\}\),
\begin{equation*}
\pi_{\ell}(N^n) \simeq \{0\},
\end{equation*}
then \(C^\infty(\overline Q^m; N^n)\) is strongly dense in \(W^{s, p}(Q^m; N^n)\).
\end{theorem}

This condition has been used by Haj\l asz \cite{Hajlasz} to give a simpler proof of Bethuel's density result for \(W^{1, p}\).
In \cite{Bousquet-Ponce-VanSchaftingen-2013}, we explain how Haj\l asz's strategy can be implemented for every Sobolev exponent \(s \ge 1\) using some pointwise estimates involving the maximal function operator inspired from the work of Maz'ya and Shaposhnikova~\cite{Mazya-Shaposhnikova}.

In order to treat the case \(s < 1\), we introduce here an additional ingredient based on the density of maps which are smooth except for a small set.
The case \(sp \ge 1\) is covered by Proposition~\ref{densityrwsp} below which relies on a projection argument due to Hardt and Lin~\cite{Hardt-Lin-1987} (Lemma~\ref{lemmaprojectionhardtlin} below) and on analytical estimates by Bourgain, Brezis and Mironescu~\cite{Bourgain-Brezis-Mironescu-2004}.
The case \(sp < 1\) is based on the density of step functions on cubes (Proposition~\ref{propositionWspConvergenceStepFunction} below) inspired by the works of Escobedo~\cite{Escobedo} and Bourgain, Brezis and Mironescu~\cite{Bourgain-Brezis-Mironescu-2000}.


\section{Strong density for \boldmath$sp \ge 1$}

The proof of Theorem~\ref{deuxiemetheorem} for \(sp \ge 1\) is based on two main ingredients: (1) when the manifold \(N^n\) is \(\floor{sp}\) simply connected, smooth maps are strongly dense in \(W^{1, q}(Q^m; N^n)\)  
for every \(1 \le q < \floor{sp} + 1\) and (2)
locally Lipschitz continuous maps outside a set of dimension \(m - \floor{sp} - 1\) are dense in \(W^{s, p}(Q^m; N^n)\).

The proof of the first assertion can be found in \cites{Hajlasz,Bousquet-Ponce-VanSchaftingen-2013}.
Before giving the precise statement of the second assertion,
we introduce for \(j \in \{0, \dots, m-2\}\) the class \(\mathcal{R}_j (Q^m; N^n)\) of maps \(u : \overline Q^m \to N^n\) such that
\begin{enumerate}[\((i)\)]
\item there exists a finite union of \(j\) dimensional submanifolds \(T \subset \R^m\) such that \(u\) is locally Lipschitz continuous in \(\overline Q^m \setminus T\),
\item for almost every \(x \in \overline Q^m \setminus T\),
\[
\abs{D u(x)}\leq \frac{C}{\dist{(x, T)}},
\]
for some constant \(C>0\) depending on \(u\).
\end{enumerate}
We observe that for every \(1 \le q < m-j\), \(\mathcal{R}_j (Q^m; N^n) \subset W^{1, q}(Q^m; N^n)\), whence by the Gagliardo-Nirenberg interpolation inequality~
\citelist{\cite{Brezis-Mironescu-2001} \cite{Mazya-Shaposhnikova-2002}*{Remark~1}}, for every \(0 < s < 1\),
\[
\mathcal{R}_j (Q^m; N^n) \subset W^{s, \frac{q}{s}}(Q^m; N^n).
\]
In particular, \(\mathcal{R}_{m - \floor{sp} - 1}(Q^m; N^n)\) is a subset of \(W^{s, p}(Q^m; N^n)\). 

Assertion \((2)\) above can be stated as follows:

\begin{proposition}\label{densityrwsp}
If \(1 \le sp < m\) and \(N^n\) is \(\floor{sp} - 1\) simply connected, then the class \(\mathcal{R}_{m - \floor{sp} - 1}(Q^m; N^n)\) is strongly dense in  \(W^{s, p}(Q^m; N^n)\).
\end{proposition}

The proof of Theorem~\ref{theoremBrezisMironescu} by Brezis and Mironescu is based on the fact that \(\mathcal{R}_{m - \floor{sp} - 1}(Q^m; N^n)\) is strongly dense in  \(W^{s,p}(Q^m; N^n)\) for \emph{every} compact manifold \(N^n\).
This is also known to be the case for every \(s \in \N_*\) \cites{Bethuel,Bousquet-Ponce-VanSchaftingen}.
A previous density result of this type for \(\S^1\) valued maps in \(W^{\frac{1}{2}, 2}\) is due to Rivière~\cite{Riviere} (see also \cite{Bourgain-Brezis-Mironescu-2004}).

We temporarily assume Proposition~\ref{densityrwsp} and complete the proof of Theorem~\ref{deuxiemetheorem}:

\begin{proof}[Proof of Theorem~\ref{deuxiemetheorem} when \(sp \ge 1\)]
By Proposition~\ref{densityrwsp}, we only need to prove that any map \(u\in \mathcal{R}_{m - \floor{sp} - 1}(Q^m; N^n)\) can be  approximated in the \(W^{s,p}\) norm by smooth maps. 

Since \(u\in W^{1, q}(Q^m; N^n)\) for every \(1 \le q < \floor{sp} + 1\), by the topological assumption on the manifold \(N^n\) there exists a sequence of smooth maps converging to \(u\) in \(W^{1, q}(Q^m; N^n)\). 
When \(sp > 1\), we may take \(q = sp\) and by the Gagliardo-Nirenberg interpolation inequality \cite{Bourgain-Brezis-Mironescu-2000}*{Lemma~D.1} the same sequence converges to \(u\) in  \(W^{s, p}(Q^m; N^n)\). 
The Gagliardo-Nirenberg interpolation inequality fails for \(q = 1\) in the sense that \(W^{1, 1} \cap L^\infty\) is not continuously imbedded into \(W^{s, \frac{1}{s}}\).
When \(sp = 1\) we then take any fixed \(1 < q < 2\) and by the Gagliardo-Nirenberg interpolation inequality \(W^{s, p}\) is continuously imbedded in \(W^{1, q}\). 
This implies that the sequence converges to \(u\) in \(W^{s, p}(Q^m; N^n)\) as before. 
\end{proof}

We now turn ourselves to the proof of Proposition~\ref{densityrwsp}.
The main geometric ingredient asserts the existence of a retraction from a cube \(Q_R^\nu\) onto \(N^n\) except for a small set \cite{Hardt-Lin-1987}*{Lemma~6.1}:

\begin{lemma}\label{lemmaprojectionhardtlin}
Let \(\ell \in \{0, \dots, \nu - 2\}\).
If \(N^n\) is \(\ell\) simply connected and contained in a cube \(Q_R^\nu\) for some \(R > 0\), then
there exist a closed subset \(X\subset Q_R^\nu \setminus N^n\) contained in a finite union of \(\nu- \ell - 2\) dimensional planes and a locally Lipschitz retraction \(\kappa : Q_R^\nu\setminus X \to N^n\) such that for  \(x\in Q_R^\nu \setminus X\),
\begin{equation*}
\abs{D\kappa (x)} \leq \frac{C}{\dist{(x, X)}},
\end{equation*}
for some constant \(C > 0\) depending on \(\nu\) and \(N^n\).
\end{lemma}

\begin{proof}
Let \(\mathcal{K}\) be a triangulation of a polyhedral neighborhood \(K^{\nu}\) of \(N^n\) such that \(N^n\) is a Lipschitz deformation retract of \(K^{\nu}\).
In particular, \(K^\nu\) and \(N^n\) are homotopically equivalent \cite{Hatcher}*{p.~3} and there exists a Lipschitz retraction \(h : K^{\nu} \to N^{n}\).
We extend \(\mathcal{K}\)  as a triangulation of \(Q_R^{\nu}\) that we denote by \(\mathcal{T}\).
Since for every \(j \in \{0, \dots, \ell\}\),  
\[
\pi_j(K^{\nu}) \simeq \pi_j(N^n) \simeq  \{0\},
\]
there exists a Lipschitz retraction \(g : T^{\ell+1} \cup K^{\nu} \to K^{\nu}\). 
Denoting by \(\mathcal{L}\) a dual skeleton of \(\mathcal{T}\) \cite{Vick}*{Chapter~6},
let \(f : \big(T^{\nu} \setminus L^{\nu-\ell-2}\big) \cup K^{\nu} \to T^{\ell+1}\cup K^{\nu} \) be a locally Lipschitz retraction such that for every \(x\in \big(T^{\nu} \setminus L^{\nu-\ell-2}\big) \cup K^{\nu}\),
\[
  \abs{Df(x)} \leq C \frac{1}{\dist( x, L^{\nu-\ell-2})}.
\]
The conclusion follows by taking 
\[
X :=  \overline{L^{\nu-\ell-2} \setminus K^{\nu}} \quad \text{and} \quad \kappa := h \circ g \circ f.\qedhere
\]
\end{proof}

The next lemma ensures that the approximation we construct in the proof of Proposition~\ref{densityrwsp} belongs to a suitable class \(\mathcal{R}_j\).

\begin{lemma}\label{lemmaconstantrank}
Let \(\Omega \subset \R^m\) be an open set, \(v \in C^{\infty} (\Omega; \R^\nu)\) and let \(\lambda \in \N\) be such that \(\lambda \leq \min{\{m, \nu\}}\). 
If \(Y \subset \R^\nu\) is a finite union of \(\nu-\lambda\) dimensional planes, then for almost every \(\xi \in \R^\nu\),
\begin{enumerate}[\((i)\)]
\item the set \(v^{-1}(Y+\xi)\) is a finite union of smooth submanifolds of \(\R^m\) of dimension \(m-\lambda\),
\item 
for every compact subset \(K \subset \Omega\) there exists a constant \(C>0\) such that for every \(x \in K\),
\[
\dist{(x , v^{-1}(Y + \xi))} \leq C \dist{(v(x), Y + \xi)}. 
\]
\end{enumerate} 
\end{lemma}
\begin{proof}
We first assume that \(Y\) is a single \(\nu - \lambda\) dimensional plane and, without loss of generality,
\begin{equation}
\label{eqChoiceY}
Y=\{0'\} \times \R^{\nu-\lambda}
\end{equation}
with \(0'\ \in \R^\lambda\). 
Let \(P : \R^\lambda \times \R^{\nu-\lambda} \to \R^\lambda\) be the orthogonal projection on the  \(\lambda\) first coordinates. 
For every \(\xi = (\xi', \xi'') \in \R^\lambda \times \R^{\nu-\lambda}\),
\[
v^{-1}(Y + \xi)
= v^{-1}(Y + (\xi', 0''))
= v^{-1}(P^{-1}(\{\xi'\}))
= (P \circ v)^{-1}(\{\xi'\}).
\]
By Sard's lemma, almost every \(\xi' \in \R^\lambda\) is a regular value of the map \(P \circ v\). 
We deduce in this case that \(v^{-1}(Y + \xi)\) is an \(m - \lambda\) smooth submanifold of \(\Omega\).

We pursue the proof of the estimate in \((ii)\) by assuming that \(\xi = 0\) and \(Y\) is of the form \eqref{eqChoiceY}
where every element of \(Y\) is a regular value of \(P\circ v\).
Given \(a \in \Omega\) such that \(v(a) \in Y\),
the linear transformation \(P\circ Dv(a)\) is surjective, whence there exist \(\delta>0\) with \(\overline{B^{m}_\delta(a)} \subset \Omega\) and a smooth diffeomorphism \(\psi : \overline{B^{m}_{\delta}(a)} \to \R^m\) such that for every \(x \in \overline{B^{m}_{\delta}(a)}\),
\begin{equation}\label{Pcircv}
P\circ v(x) = P\circ  Dv(a)[\psi(x)].
\end{equation}
This is a consequence of the Inverse function theorem.	
Indeed, let \(\psi_1\)  be the orthogonal projection in \(\R^m\) onto \( \ker P\circ Dv(a)\) and let \(\psi_2 = (P\circ Dv(a)|_{(\ker P\circ Dv(a))^{\perp}})^{-1} \circ P\circ v\). 
Then, \(D(\psi_1+\psi_2)(a) = \textrm{ id}_{\R^m}\), whence by the Inverse function theorem the function \(\psi=\psi_1+\psi_2\) is a smooth diffeomorphism in a neighborhood of \(a\) and satisfies \(P\circ v = P\circ Dv(a) \circ \psi\).

It follows from \eqref{Pcircv} that  \(\dist{(v(x), Y)}= \dist{(Dv(a)(\psi (x)), Y)}\).
Denoting by
\[
V = (Dv(a))^{-1}(Y),
\]
we observe that for every \(y \in B^m_\delta(a)\), \(v(y) \in Y\) if and only if \(\psi(y) \in V\).
Since \(\psi\) is a diffeomorphism, there exist \({\NewConstant} > 0\) such that for \(x \in B^m_\delta(a)\),
\[
\dist{(x, v^{-1}(Y) \cap B^m_\delta(a))}
\le {\SameConstant} \dist{(\psi(x), V \cap \psi(B^m_\delta(a)))}.
\]
By the counterpart of \((ii)\) for linear transformations, there exists a constant \({\Constant} > 0\) such that for every \(z \in \R^m\),
\[
\dist{(z, V)}
\le {\SameConstant} \dist{(Dv(a)[z], Y)};
\]
this property can be proved using the linear bijection \(R \circ Dv(a)|_{V^\perp}\), where \(R\) is the orthogonal projection onto \(Y^\perp\).
Thus, for every \(x \in B^m_\delta(a)\),
\[
\dist{(\psi(x), V)}
\le {\SameConstant} \dist{(Dv(a)[\psi(x)], Y)}
= {\SameConstant} \dist{(v(x), Y)}.
\]
To conclude the argument, take \(0 < \underline{\delta} \le \delta\) such that for every \(x \in B^{m}_{\underline{\delta}}(a)\),
\[
\dist{(x , v^{-1}(Y))} 
= \dist{(x, v^{-1}(Y) \cap B^{m}_{\delta}(a))}
\]  
and
\[
\dist{(\psi(x), V)} 
= \dist{(\psi(x), V \cap \psi(B_\delta^m(a)))}.
\]
We deduce from the above that for \(x \in B^{m}_{\underline{\delta}}(a)\),
\[
\dist{(x , v^{-1}(Y))} 
\le \NewConstant\Constant \dist{(v(x), Y)}. 
\]
Using a covering argument of \(K \cap v^{-1}(Y)\), the conclusion follows when \(Y\) is a single \(\nu - \lambda\) dimensional plane.

We now assume that \(Y\) is a finite union of \(\nu - \lambda\) dimensional planes \(Y_1, \dots, Y_j\).
The first assertion is true for almost every \(\xi \in \R^\nu\).
Concerning the second assertion, note that for every \(x \in \Omega\) and for every \(\xi \in \R^\nu\),
\[
\dist{(x , v^{-1}(Y + \xi))} = \min_{i \in \{1, \dots, j\}}{\dist{(x , v^{-1}(Y_i + \xi))}}
\]
and
\[
\dist{(v(x), Y + \xi)} = \min_{i \in \{1, \dots, j\}}{\dist{(v(x), Y_i + \xi)}}.
\]
Let \(\xi \in \R^\nu\). If the estimate holds for every \(Y_i\) with some constant \(C_i' > 0\), then for every \(x \in K\),
\[
\begin{split}
\dist{(x , v^{-1}(Y + \xi))}
& \leq  \Big(\max_{i \in \{1, \dots, j\}} {C_i'}\Big) \min_{i \in \{1, \dots, j\}} \dist{(v(x), Y_i + \xi)}\\
& =  \Big(\max_{i \in \{1, \dots, j\}} {C_i'}\Big) \dist{(v(x), Y + \xi)}. 
\end{split}
\]
This concludes the proof of the lemma.
\end{proof}

Given a domain \(\Omega \subset \R^m\) and a measurable function  \(u : \Omega \to \R^\nu\), we now estimate the convolution function \(\varphi_t * u\) and its derivative in terms of a fractional derivative of \(u\). 
More precisely, given \(0 < s < 1\) and \(1 \le p < +\infty\), define for \(x \in \Omega\) \cite{Mazya-Shaposhnikova-2002},
\[
\displaystyle D^{s, p}u(x) = \bigg(\int\limits_{\Omega} \frac{\abs{u(x) - u(y)}^p}{\abs{x - y}^{m + s p}} \dif y\bigg)^{1/p}.
\]
We assume that \(\varphi : \R^m \to \R\) be a mollifier. 
In other words,
\[
\label{eqTagMollifier}
\varphi \in C_c^\infty(B_1^m), \quad
\varphi \ge 0 \text{ in \(B_1^m\)}
\quad
\text{and}
\quad
\int\limits_{B_1^m} \varphi = 1.
\]
For every \(t > 0\), define \(\varphi_t : \R^m \to \R\) for \(h \in \R^m\) by
\[
\varphi_t(h) = \frac{1}{t^m} \varphi \Big( \frac{h}{t} \Big).
\]

Using the notation above we have the following: 

\begin{lemma}\label{lemmaestimationconvolution} 
If \(u \in W^{s, p}(\Omega; \R^\nu)\), then for every \(t > 0\) and for every \(x \in \Omega\) such that \(\dist{(x, \partial\Omega)} > t\),
\begin{enumerate}[\((i)\)]
\item \(\abs{\varphi_t * u(x) - u(x) } \leq C t^s D^{s, p} u(x)\),
\item \( \abs{D (\varphi_t * u)(x)} \leq C' t^{-(1-s)} D^{s,p}u(x)\),
\end{enumerate}
for some constants \(C > 0\) depending on \(\varphi\) and \(C' > 0\) depending on \(D\varphi\) and \(p\).
\end{lemma}

\begin{proof}
By Jensen's inequality,
\[
\begin{split}
\abs{\varphi_t * u(x)-u(x)}^p 
& \leq \int\limits_{\R^m} \varphi_t (h) \abs{u(x-h)-u(x)}^p \dif h\\
& = \int\limits_{\R^m} \varphi_t (h) \abs{h}^{m+sp} \frac{\abs{u(x-h)-u(x)}^p}{\abs{h}^{m+sp}} \dif h.
\end{split}
\]
Since \(\varphi_t\) is supported in \(B_t^m\), for every \(h \in \R^m\), \(\varphi_t(h) \abs{h}^{m+sp}\leq \NewConstant t^{sp}\). 
The first inequality follows.

Next, since \(\int_{\R^m} D \varphi_t = 0\),
\[
\abs{D (\varphi_t * u)(x)} \leq  \int\limits_{\R^m} \abs{D \varphi_t(h)} {\abs{u(x-h)-u(x)}} \dif h.
\]
Since 
\[
\int\limits_{\R^m} \abs{D\varphi_t} \le \frac{\Constant}{t},
\]
by Jensen's inequality,
\[
\begin{split}
\abs{D(\varphi_t * u)(x)}^p 
& \leq \frac{\SameConstant^{p - 1}}{t^{p-1}}
\int\limits_{\R^m} \abs{D\varphi_t(h)} \abs{u(x-h)-u(x)}^p \dif h\\
& = \frac{\SameConstant^{p - 1}}{t^{p-1}}
\int\limits_{\R^m} \abs{D\varphi_t(h)}\abs{h}^{m + sp} \frac{\abs{u(x-h)-u(x)}^p}{\abs{h}^{m + sp}} \dif h.
\end{split}
\]
Since for every \(h \in \R^m\), \(\abs{D\varphi_t(h)}|h|^{m + sp} \le \Constant t^{sp - 1}\), the second estimate follows.
\end{proof}

If \(u \in W^{s, p} (\Omega; \R^\nu)\) and \(\kappa : \R^\nu \to \R^\nu\) is Lipschitz continuous, then \(\kappa \circ u \in W^{s, p} (\Omega; \R^\nu)\) and 
\begin{equation}
\label{eqEstimateComposition}
[\kappa \circ u]_{W^{s, p}(\Omega)} 
\le \abs{\kappa}_{\Lip(\R^\nu)} [u]_{W^{s, p}(\Omega)},
\end{equation}
where \(\abs{\kappa}_{\Lip(\R^\nu)}\) denotes the best Lipschitz constant of \(\kappa\).
The next lemma gives the continuity of the composition operator \(u \mapsto \kappa \circ u\) in \(W^{s, p}\):

\begin{lemma}
\label{lemmaContinuityLipschitz}
Let \(\Omega \subset \R^m\) be a bounded open set and \(u \in W^{s, p} (\Omega; \R^\nu)\). 
For every \(\epsilon > 0\), there exists \(\delta > 0\) such that if \(\kappa : \R^\nu \to \R^\nu\) is Lipschitz continuous, \(v \in W^{s, p} (\Omega)\) and \(\norm{u - v}_{W^{s, p} (\Omega; \R^\nu)} \le \delta\), then 
\[
[\kappa \circ u - \kappa \circ v]_{W^{s, p}(\Omega)}
 \le \abs{\kappa}_{\Lip(\R^\nu)} \epsilon.
\]
\end{lemma}

By a result of Marcus and Mizel~\cite{MarcusMizel1979}*{Theorem~1} in the scalar case \(\nu = 1\), the map \(u \in W^{1, p} (\Omega; \R) \mapsto \kappa \circ u \in W^{1, p} (\Omega; \R)\) is continuous.
Lemma~\ref{lemmaContinuityLipschitz} has been proved by Bourgain, Brezis and Mironescu 
~\cite{Bourgain-Brezis-Mironescu-2004}*{Claim (5.43)}.
For the convenience of the reader we present their proof, organized differently.

\begin{proof}[Proof of Lemma~\ref{lemmaContinuityLipschitz}]
For \(u, v \in W^{s, p} (\Omega; \R^\nu)\) and \(\kappa : \R^\nu \to \R^\nu\), define for  \(x, y \in \Omega\),
\[
 I (x, y) = \frac{\abs{\kappa (u (x)) - \kappa (v (x)) - \kappa (u (y)) + \kappa (v (y))}^{p}}{\abs{x - y}^{m + sp}},
\]
so that
\[
[\kappa \circ u - \kappa \circ v]_{W^{s, p}(\Omega)}
= \int\limits_{\Omega} \int\limits_{\Omega} I(x, y) \dif x \dif y.
\]
Observe that
\[
\begin{split}
 I (x, y) 
 &\le 2^{p - 1} \frac{\abs{\kappa (u (x)) - \kappa (v (x))}^p + \abs{\kappa (u (y)) - \kappa (v (y))}^{p}}{\abs{x - y}^{m + sp}}\\
 &\le 2^{p - 1} \abs{\kappa}_{\Lip(\R^\nu)}^p \frac{\abs{u (x) - v (x)}^p + \abs{u (y) - v (y)}^{p}}{\abs{x - y}^{m + sp}}
\end{split}
\]
and that 
\[
\begin{split}
 I (x, y) 
 &\le 2^{p - 1} \frac{\abs{\kappa (u (x)) - \kappa (u (y))}^p + \abs{\kappa (v (x)) - \kappa (v (y))}^{p}}{\abs{x - y}^{m + sp}}\\
 &\le 2^{p - 1} \abs{\kappa}_{\Lip(\R^\nu)}^p \frac{\abs{u (x) - u (y)}^p + \abs{v (x) - v (y)}^{p}}{\abs{x - y}^{m + sp}}\\
 &\le {\NewConstant} \abs{\kappa}_{\Lip(\R^\nu)}^p 
 \bigg( \frac{\abs{u (x) - u (y)}^p}{\abs{x - y}^{m + sp}} + \frac{\abs{u (x) - v(x) - u (y) + v(y)}^p}{\abs{x - y}^{m + sp}} \bigg).
\end{split}
\]

Given \(\epsilon > 0\), let 
\[
A_{v, \epsilon} 
= \Bigl\{ (x, y) \in \Omega \times \Omega : \abs{u (x) - v (x)}^p + \abs{u (y) - v (y)}^{p} \ge \epsilon \abs{x - y}^{m + sp}\Bigr\}.
\]
Using the first upper bound of \(I(x, y)\) on the set \((\Omega \times \Omega) \setminus A_{v, \epsilon}\) and the second one on the set \(A_{v, \epsilon}\), we get
\begin{multline*}
[\kappa\circ u - \kappa\circ v]^p_{W^{s, p}(\Omega)}\\
\le \abs{\kappa}_{\Lip(\R^\nu)}^p 
\biggl(2^{p - 1} \epsilon \abs{\Omega}^2
 + {\SameConstant} \iint\limits_{A_{v, \epsilon}} \frac{\abs{u (x) - u (y)}^p}{\abs{x - y}^{m + sp}} \dif x \dif y + {\SameConstant} [u - v]_{W^{s, p}(\Omega)}^p\biggr).
\end{multline*}
Since \(u \in W^{s, p}(\Omega)\) and \(\abs{A_{v, \epsilon}} \to 0\) as \(v \to u\) in \(W^{s, p} (\Omega)\), the conclusion follows from the Dominated convergence theorem.
\end{proof}

Despite of the estimate \eqref{eqEstimateComposition}, when \(\kappa\) is not affine there is no inequality of the form
\[
[\kappa \circ u - \kappa \circ v]_{W^{s, p}(\Omega)} 
\le C \abs{\kappa}_{\Lip(\R^\nu)} [u - v]_{W^{s, p}(\Omega)}.
\]
In fact, the map \(u \mapsto \kappa \circ u \) is not even uniformly continuous in \(W^{s, p}\).
We explain the argument when the domain is the unit cube \(Q^m\).
For this purpose, let \(\varphi \in C^\infty_c (Q^m; \R^\nu)\) and  denote by \(\Bar{\varphi}\)  the periodic extension of \(\varphi\) to \(\R^m\).
Define for \(j \in \N_*\),
\[
v_j (x) = \Bar{\varphi} (j x)
\]
and, for some fixed \(\xi \in \R^\nu\),
\[
u_j (x) = \Bar{\varphi} (j x) + \xi.
\]
We observe that 
\[
\norm{u_j - v_j}_{W^{s, p}(Q^m)} 
= \norm{u_j - v_j}_{L^{p}(Q^m)} 
= 2^m \abs{\xi}
\]
whereas 
\begin{multline}
\label{eqEstimationSeminormeGagliardo}
[\kappa \circ u_j - \kappa \circ v_j]_{W^{s, p}(Q^m)}^p \\
\ge  j^{sp} \int\limits_{Q^m}\int\limits_{Q^m} \frac{\bigabs{\kappa (\varphi (x) + \xi) - \kappa(\varphi (x)) - \kappa (\varphi (y) + \xi) + \kappa (\varphi (y))}^p}{\abs{x - y}^{m + sp}} \dif x \dif y.
\end{multline}
When \(\kappa\) is not affine, there exist \(\xi, \tau, \sigma \in \R^\nu\) such that 
\[
 \kappa (\tau + \xi) - \kappa (\tau) 
 \ne \kappa (\sigma + \xi) - \kappa (\sigma).
\]
Taking \(\varphi \in C^\infty_c (Q^m; \R^\nu)\) for which both sets 
\(\varphi^{-1}(\{\sigma\})\) and \(\varphi^{-1}(\{\tau\})\) have positive measure, we have
\[
\int\limits_{Q^m} \int\limits_{Q^m} \frac{\bigabs{\kappa (\varphi (x) + \xi) - \kappa(\varphi (x)) - \kappa (\varphi (y) + \xi) + \kappa (\varphi (y))}^p}{\abs{x - y}^{m + sp}} \dif x \dif y  > 0. 
\]
As we let \(j\) tend to infinity in \eqref{eqEstimationSeminormeGagliardo}, we conclude that \(u \mapsto \kappa \circ u\) is not uniformly continuous in \(W^{s, p}\).

\begin{proof}[Proof of Proposition~\ref{densityrwsp}]
Let \(u\in W^{s, p}(Q^m; N^n)\). 
The restrictions to \(Q^m\) of the maps \(u_\gamma \in W^{s, p}(Q_{1 + 2\gamma}^m; N^n)\) defined for \(x \in Q_{1 + 2 \gamma}^m\) by \(u_\gamma(x) = u (x/(1 + 2 \gamma))\)  converge strongly to \(u\) in \(W^{s, p}(Q^m; N^n)\) as \(\gamma\) tends to \(0\). 
We can thus assume from the beginning that \(u \in W^{s, p}(Q_{1 + 2\gamma}^m; N^n)\) for some \(\gamma>0\).

Let \(\kappa : \R^\nu \setminus X \to N^n\) be the locally Lipschitz retraction of Lemma~\ref{lemmaprojectionhardtlin} with \(\ell=\floor{sp}-1\); we may assume that \(\nu \ge \ell + 2\).
For every \(\xi \in \R^\nu\), we consider the map \(\kappa_\xi : \R^\nu \setminus (X + \xi) \to N^n\) defined by
\[
\kappa_\xi(x) = \kappa(x - \xi).
\]

Given a mollifier \(\varphi\) (see p.~\pageref{eqTagMollifier} above), the map \(\kappa_\xi \circ (\varphi_t * u)\) is locally Lipschitz continuous in \(Q_{1 + \gamma}^m\setminus (\varphi_t * u)^{-1}(X + \xi)\). 
Moreover, by the chain rule and by the pointwise estimate satisfied by \(D\kappa\), 
\begin{equation}
\label{estimateduea}
\bigabs{D[\kappa_\xi \circ (\varphi_t * u)]}
\leq {\NewConstant} \frac{\abs{D (\varphi_t * u)}}{\dist{(\varphi_t * u, X+ \xi)}}.
\end{equation}

The set \(X\) is contained in a finite union of \(\nu-\floor{sp}-1\) dimensional planes \(Y\) in \(\R^\nu\). 
Applying Lemma~\ref{lemmaconstantrank} to \(v= \varphi_t * u \in C^{\infty}(Q^{m}_{1+\gamma} ; \R^\nu)\), we obtain that for every \(0 < t \le \gamma\) and for almost every \(\xi \in \R^{\nu}\), the set \((\varphi_t * u)^{-1}(X + \xi)\) is contained in a finite union of \(m-\floor{sp}-1\) dimensional submanifolds,
\[
T = {(\varphi_t * u)^{-1}(Y+\xi)}.
\] 
By \eqref{estimateduea} and the inclusion  \(X \subset Y\), 
\[
\bigabs{D[\kappa_\xi \circ (\varphi_t * u)]}
\leq \Constant \frac{1}{\dist{(\varphi_t * u, X + \xi)}}
\leq \SameConstant \frac{1}{\dist{(\varphi_t * u, Y + \xi)}}.
\]
By the second part of Lemma~\ref{lemmaconstantrank}, we conclude that for \(x \in \overline{Q}^m \setminus (\varphi_t * u)^{-1}(Y + \xi)\),
\[
\bigabs{D[\kappa_\xi \circ (\varphi_t * u)](x)} 
\leq \Constant \frac{1}{\dist{(x, (\varphi_t * u)^{-1}(Y + \xi))}} = \frac{\SameConstant}{\dist{(x, T)}}. 
\]
In particular, for every \(0 < t \le \gamma\) and for almost every \(\xi \in \R^\nu\), the map \(\kappa_\xi \circ (\varphi_t * u)\) belongs to \(\mathcal{R}_{m-\floor{sp}-1}(Q^m; N^n)\).

\medskip
We proceed using an idea from \cite{Bourgain-Brezis-Mironescu-2004} for \(W^{\frac{1}{2}, 2}\) maps with values into the circle \(\S^1\).
Let 
\[
\alpha= \frac{1}{4}\dist{(X, N^n)},
\]
let \(\theta : \R^\nu \to \R\) be a Lipschitz continuous function such that 
\begin{enumerate}[\((a)\)]
\item for \(\dist{(x, X)} \le 2\alpha\), \(\theta(x) = 1\),
\item for \(\dist{(x, X)} \ge 3\alpha\), \(\theta(x) = 0\),
\end{enumerate}
and let 
\[
 \Bar{\kappa}_\xi  = (1 - \theta) \kappa_{\xi} \quad \text{and} \quad \underline{\kappa}_\xi = \theta \kappa_{\xi}.
\]
Since \(\kappa_{\xi} = \Bar{\kappa}_{\xi}\) on \(u (Q^{m}_{1 + 2 \gamma}) \subset N^n\), we have by the triangle inequality,
\begin{multline}
\label{eqTriangleInequality}
\bignorm{\kappa_\xi \circ (\varphi_t * u) - u}_{W^{s, p}(Q^m)}\\
\le
\norm{\underline{\kappa}_\xi \circ (\varphi_t * u)}_{W^{s, p}(Q^m)} + 
\norm{\Bar{\kappa}_{\xi} \circ (\varphi_t * u) - \Bar{\kappa}_\xi \circ u}_{W^{s, p}(Q^m)}\\ + \norm{\kappa_{\xi} \circ u - u}_{W^{s, p}(Q^m)}.
\end{multline}
Since \(\kappa\) is Lipschitz continuous on a neighborhood of \(N^n\) and \(\kappa_\xi \circ u = \kappa(u - \xi)\), we have by continuity of the composition operator in \(W^{s, p}\) (Lemma~\ref{lemmaContinuityLipschitz}),
\begin{equation}
\label{eqLimitThirdTerm}
\lim\limits_{\xi \to 0}{\norm{\kappa_{\xi} \circ u - u}_{W^{s,  p}(Q^m)}}=0.
\end{equation}
By Lemma~\ref{lemmaContinuityLipschitz}, as the maps \(\Bar{\kappa}_{\xi}\) are uniformly Lipschitz continuous and \(\varphi_t * u\) converges to \(u\) in \(W^{s, p} (Q^m)\),
\begin{equation}
\label{eqLimitSecondTerm}
 \lim_{t \to 0} \norm{\Bar{\kappa}_{\xi} \circ (\varphi_t * u) - \Bar{\kappa}_{\xi} \circ u}_{W^{s, p}(Q^m)} = 0,
\end{equation}
uniformly with respect to \(\xi\).

It remains to estimate the first term in the right hand side of \eqref{eqTriangleInequality}. 
This is done in the following:

\begin{Claim}\label{lemmacvgceuea}
For every \(0 < t \le \gamma\),
\[
\int\limits_{B^\nu_{\alpha}} \|\underline{\kappa}_\xi \circ (\varphi_t * u)\|_{W^{s, p}(Q^m)}^p \dif\xi
\le 
C \int\limits_{\{\abs{\varphi_t * u - u} \ge \alpha\}} (D^{s, p}u)^p
\]
\end{Claim}

We assume temporarily the claim, and complete the proof of Proposition~\ref{densityrwsp}.
Since \(D^{s, p} u \in L^p(Q^m)\) and \(\varphi_t * u\) converges to \(u\) in measure as \(t\) tends to zero, by the claim we have
\[
\lim_{t \to 0}{\int\limits_{B_\alpha^\nu}
\norm{\underline{\kappa}_\xi \circ (\varphi_t * u)}_{W^{s, p}(Q^m)}^p 
\dif\xi
} = 0.
\]
By the Chebyshev inequality, 
\begin{multline*}
\lim_{t \to 0}
\biggabs{\biggl\{\xi \in B_\alpha^m : \norm{\underline{\kappa}_\xi \circ (\varphi_t * u)}_{W^{s, p}(Q^m)}^p \ge \biggl( \int\limits_{B_\alpha^m}
\norm{\underline{\kappa}_\zeta \circ (\varphi_t * u)}_{W^{s, p}(Q^m)}^p 
\dif\zeta
\biggr)^\frac{1}{2} \biggr\}}\\
 = 0.
\end{multline*}
Thus, for every \(0 < t \le \gamma\), there exists \(\xi_t \in B_\alpha^m\) such that \(\lim\limits_{t \to 0}{\xi_t} = 0\) and
\[
\lim_{t \to 0}\norm{\underline{\kappa}_{\xi_t} \circ (\varphi_t * u)}_{W^{s, p}(Q^m)} = 0.
\]
We conclude from \eqref{eqTriangleInequality}, \eqref{eqLimitThirdTerm} and \eqref{eqLimitSecondTerm} that
\[
\lim_{t \to 0}{\bignorm{\kappa_{\xi_t} \circ (\varphi_t * u) - u}_{W^{s, p}(Q^m)}} = 0.
\]
This gives the conclusion of Proposition~\ref{densityrwsp}.
\end{proof}

It remains to establish the claim:

\begin{proof}[Proof of the claim]
Let \(1 < q < p < r\) be such that 
\begin{equation}
\label{eqRelationExponents}
\frac{1}{p} 
= \frac{1-s}{r} + \frac{s}{q}.
\end{equation}
By the Gagliardo-Nirenberg interpolation inequality,
\begin{equation}
\label{eqEstimateGagliardoNirenberg}
\norm{\underline{\kappa}_\xi \circ (\varphi_t * u)}_{W^{s, p}(Q^m)} 
\leq {\Constant} \|\underline{\kappa}_\xi \circ (\varphi_t * u)\|_{L^{r}(Q^m)}^{1-s} \|\underline{\kappa}_\xi \circ (\varphi_t * u)\|_{W^{1, q}(Q^m)}^s. 
\end{equation}
As \(N^n\) is compact, we observe that the functions \(\underline{\kappa}_\xi \circ (\varphi_t * u)\) are uniformly bounded and supported on the set \(\big\{\dist{(\varphi_t * u, X)} \le 3\alpha\big\}\).
Moreover,
\[
\big\{\dist{(\varphi_t * u, X)} \le 3\alpha \big\} 
\subset \big\{|\varphi_t * u - u| \geq \alpha \big\}.
\]
Thus,
\begin{equation}
\label{eqLrestimate}
\|\underline{\kappa}_\xi \circ (\varphi_t * u)\|_{L^r(Q^m)}
\leq {\SameConstant} \bigabs{\{|\varphi_t * u - u| \geq \alpha\}}^\frac{1}{r}.
\end{equation}
Next, by the Leibniz rule and by \eqref{estimateduea},
\begin{equation*}
\begin{split}
\abs{D (\underline{\kappa}_\xi \circ (\varphi_t * u))} 
& \leq \Big( \abs{D\theta(\varphi_t * u)} \abs{\kappa_\xi (\varphi_t * u)} + \abs{\theta(\varphi_t * u)} \abs{D\kappa_\xi (\varphi_t * u)} \Big) \abs{D(\varphi_t * u)}\\
& \leq {\Constant} \bigg(1 + \frac{1}{\dist{(\varphi_t * u, X + \xi)}} \bigg)  
\abs{D(\varphi_t * u)}.
\end{split}
\end{equation*}
Since the functions \(D (\underline{\kappa}_\xi \circ (\varphi_t * u))\) are also supported in the set \(\big\{|\varphi_t * u - u| \geq \alpha \big\}\), we get
\begin{multline*}
\norm{\underline{\kappa}_\xi \circ (\varphi_t * u)}_{W^{1, q}(Q^m)}^q\\
\le
{\Constant} \int\limits_{\{|\varphi_t * u - u| \geq \alpha\}}
\bigg[ 1 + \bigg(1 + \frac{1}{\dist{(\varphi_t * u, X + \xi)}^q} \bigg) \abs{D(\varphi_t * u)}^q \bigg].
\end{multline*}
For
\[
\boxed{q \ge sp,}
\]
we have by Hölder's inequality and by Fubini's theorem,
\begin{multline*}
\int\limits_{B_\alpha^\nu}
\norm{\underline{\kappa}_\xi \circ (\varphi_t * u)}_{W^{1, q}(Q^m)}^{sp}
\dif\xi\\
\begin{aligned}
& \le \abs{B_\alpha^\nu}^{1 - \frac{sp}{q}} \Biggl(\; \int\limits_{B_\alpha^\nu}
\norm{\underline{\kappa}_\xi \circ (\varphi_t * u)}_{W^{1, q}(Q^m)}^{q} 
\dif\xi \Biggr)^\frac{sp}{q}\\
& \le {\Constant}\Biggl(\; \int\limits_{\{|\varphi_t * u - u| \geq \alpha\}} 
\int\limits_{B_\alpha^\nu}
\bigg[ 1 + \bigg(1 + \frac{1}{\dist{(\varphi_t * u(x), X + \xi)}^q} \bigg) \abs{D(\varphi_t * u)(x)}^q \bigg] \dif \xi \dif x \Biggr)^{\frac{sp}{q}}.
\end{aligned}
\end{multline*}
We have
\[
\begin{split}
\int\limits_{B_\alpha^\nu}
\frac{1}{\dist{(\varphi_t * u(x), X +\xi)}^q}
\dif \xi
& =  \int\limits_{B_\alpha^\nu}
\frac{1}{\dist{(\varphi_t * u(x) - X, \xi)}^q}
\dif \xi \\
& =  \int\limits_{B_\alpha^\nu+\varphi_t*u(x)}
\frac{1}{\dist{(X, \xi)}^q} \dif \xi \\
& \leq  \int\limits_{B_R^\nu}
\frac{1}{\dist{(X, \xi)}^q} \dif \xi ,
\end{split}
\]
where \(R > 0\) is such that for every \(x \in Q^m\), \(B_\alpha^\nu+\varphi_t*u(x) \subset B_R^\nu\).
Since \(X\) is a closed subset of a finite union of \(\nu - \floor{sp} - 1\) dimensional planes, assuming in addition that
\[
\boxed{q < \floor{sp} + 1,}
\]
then the last integral is finite.
Thus,
\[
\int\limits_{B_\alpha^\nu}
\norm{\underline{\kappa}_\xi \circ (\varphi_t * u)}_{W^{1, q}(Q^m)}^{sp}
\dif\xi
\le {\Constant} \biggl(\int\limits_{\{|\varphi_t * u - u| \geq \alpha\}} 
\big[1 + \abs{D(\varphi_t * u)}^q \big] \biggr)^{\frac{sp}{q}}.
\]
Inserting this estimate and \eqref{eqLrestimate} into \eqref{eqEstimateGagliardoNirenberg}, we deduce that
\begin{multline*}
\int\limits_{B_\alpha^\nu}
 \|\underline{\kappa}_\xi \circ (\varphi_t * u)\|_{W^{s, p}(Q^m)}^p \dif \xi\\
\leq {\Constant} \bigabs{\{|\varphi_t * u - u| \geq \alpha\}}^\frac{(1-s)p}{r} \biggl(\int\limits_{\{|\varphi_t * u - u| \geq \alpha\}}
\big[ 1 + \abs{D(\varphi_t * u)}^q \big] \biggr)^{\frac{sp}{q}}. 
\end{multline*}
Since \(q < p\), by Hölder's inequality and by the identity \eqref{eqRelationExponents} satisfied by the exponents \(r\), \(p\) and \(q\),
\begin{multline*}
\int\limits_{B_\alpha^\nu}
\|\underline{\kappa}_\xi \circ (\varphi_t * u)\|_{W^{s, p}(Q^m)}^p 
\dif\xi\\
\leq {\Constant} \bigabs{\{|\varphi_t * u - u| \geq \alpha\}}^{1-s} \biggl(\int\limits_{\{|\varphi_t * u - u| \geq \alpha\}}
\big[ 1 + \abs{D(\varphi_t * u)}^p \big] \biggr)^s.
\end{multline*}
By the Chebyshev inequality and by Lemma~\ref{lemmaestimationconvolution},
\[
\begin{split}
\bigabs{\{|\varphi_t * u - u| \geq \alpha\}}
& \le
\frac{1}{\alpha^p} \int\limits_{\{|\varphi_t * u - u| \geq \alpha\}} |\varphi_t * u - u|^p\\
& \le 
{\Constant} t^{sp} \int\limits_{\{|\varphi_t * u - u| \geq \alpha\}} (D^{s, p} u)^p.
\end{split}
\]
By Lemma~\ref{lemmaestimationconvolution}, we also have
\[
\int\limits_{\{|\varphi_t * u - u| \geq \alpha\}}
\abs{D(\varphi_t * u)}^p
\le \frac{\Constant} {t^{(1 - s)p}} \int\limits_{\{|\varphi_t * u - u| \geq \alpha\}} (D^{s, p} u)^p.
\]
We conclude that
\[
\int\limits_{B_\alpha^\nu}
\|\underline{\kappa}_\xi \circ (\varphi_t * u)\|_{W^{s, p}(Q^m)}^p 
\dif\xi
\le {\Constant} (t^{sp} + 1) \int\limits_{\{|\varphi_t * u - u| \geq \alpha\}} (D^{s, p} u)^p.
\]
This proves the claim.
\end{proof}


\section{Strong density for \boldmath$sp < 1$}

The proof of Theorem~\ref{deuxiemetheorem} when \(sp < 1\) relies on the density of step functions in \(W^{s, p}\) based on a Haar projection~\cite{Bourgain-Brezis-Mironescu-2000}.
This analytical step is developped in Propositions~\ref{propositionEstimateELp} and~\ref{propositionWspConvergenceStepFunction} below.
Then, a standard tool from Differential topology (Proposition~\ref{propositionFiniteSet}) allows us to reduce the problem to an approximation of a map with values in a convex set and this can be carried out by convolution.

Given a function \(v \in L^1(Q^m ; \R^\nu)\), we consider the Haar projection \(E_j(v) : Q^m \to  \R^\nu\) defined almost everywhere on \(Q^{m}\).
More precisely, denoting by \(K^{m}_{2^{-j}}\) the standard cubication of \(Q^m\) in \(2^{jm}\) cubes of radius \(2^{-j}\), for every \(\sigma \in K^{m}_{2^{-j}}\) the function 
\(E_j(v)\) is constant in \(\Int{\sigma}\) and for \(x \in \Int{\sigma}\),
\[
  E_j(v)(x)= \frac{1}{|\sigma|} \int\limits_{\sigma} v.
\]
In particular, \(E_j(v)\) is a step function.

\begin{proposition}
\label{propositionEstimateELp}
Let \(v \in L^p(Q^m; \R^\nu)\).
Then, for every \(j \in \N_*\),
\[
\norm{E_j(v)}_{L^p(Q^m)} \leq \norm{v}_{L^p(Q^m)} 
\]
and the sequence \((E_j(v))_{j \in \N_*}\) converges strongly to \(v\) in \(L^p(Q^m; \R^\nu)\). 
\end{proposition}

\begin{proof}
The estimate follows from H\"older's inequality.
To prove the convergence of the sequence \((E_j(v))_{j \in \N_*}\), we write
\[
\begin{split}
\norm{E_j(v) - v}_{L^{p}(Q^m)}^p
& = \sum_{\sigma \in K^{m}_{2^{-j}}} \int\limits_{\sigma}\bigabs{v(x)-\frac{1}{|\sigma|}\int\limits_{\sigma}v}^p \dif x \\
& \leq \sum_{\sigma \in K^{m}_{2^{-j}}} \frac{1}{|\sigma|} \int\limits_{\sigma}\int\limits_{\sigma}\abs{v(x)-v(y)}^p \dif x \dif y.
\end{split}
\]
Approximating \(v\)  in \(L^p(Q^m; \R^\nu)\) by a continuous function, we deduce that the right-hand side converges to \(0\) as \(j\) tends to infinity.
This gives the conclusion.
\end{proof}

The counterpart of the previous proposition still holds in the case of fractional Sobolev spaces \(W^{s, p}\) for \(sp < 1\)
and is due to Bourgain, Brezis and Mironescu~\cite{Bourgain-Brezis-Mironescu-2000}*{Corollary~A.1}:

\begin{proposition}
\label{propositionWspConvergenceStepFunction}
Let \(v \in W^{s, p}(Q^m; \R^\nu)\). If \(sp < 1\), then for every \(j \in \N_*\),
\[
[E_j(v)]_{W^{s, p}(Q^m)} \leq C [v]_{W^{s, p}(Q^m)} 
\]
for some constant \(C > 0\) depending on \(s\), \(p\) and \(m\).
In addition, the sequence \((E_j(v))_{j \in \N_*}\) converges strongly to \(v\) in \(W^{s, p}(Q^m; \R^\nu)\). 
\end{proposition}

The proof of Bourgain, Brezis and Mironescu is based on a characterization of the fractional Sobolev spaces \(W^{s, p}\) for \(sp < 1\) due to Bourdaud~\cite{Bourdaud-1995} in terms of the Haar basis.
We present an alternative argument relying directly on the Gagliardo seminorm.
The main ingredient is the following:

\begin{Claim}
If \(sp < 1\), then for every \(\sigma, \rho \in K^{m}_{2^{-j}}\),
\[
\int\limits_\sigma \int\limits_\rho \frac{1}{\abs{x - y}^{m + sp}} \dif x \dif y 
\le C' \frac{\abs{\sigma}\abs{\rho}}{\delta (\sigma, \rho)^{m + sp}},
\]
where
\[
\delta(\sigma, \rho)
= \sup{\bigl\{ \abs{x - y} : x \in \sigma\ \text{and}\ y \in \rho \bigr\}}
\]
and the constant \(C' > 0\) depends on \(m\) and \(sp\).
\end{Claim}
\begin{proof}[Proof of the claim] For every \((x, y) \in \sigma\times \rho\), 
\[
\abs{x-y}\geq \delta(\sigma, \rho) -\textrm{ diam } \sigma -\textrm{ diam } \rho =  \delta(\sigma, \rho) - 2^{-j+2}\sqrt{m}.
\]
If \(\delta(\sigma, \rho) \geq 2^{-j+3} \sqrt{m} \), then \(\frac{1}{2} \delta(\sigma, \rho) \leq |x-y| \leq \delta(\sigma, \rho)\) and the result follows in this case. 
Since the indicator function of the unit cube \(\chi_{Q^m}\) belongs to \( W^{s,p}(\R^m)\) for \(sp< 1\), a scaling argument leads to the following estimate
\[
\frac{1}{\abs{\sigma}\abs{\rho}}\int\limits_\sigma \int\limits_\rho \frac{1}{\abs{x - y}^{m + sp}} \dif x \dif y  \le \NewConstant 2^{j(m+sp)}.
\]
In turn, this implies the claim when  \(\delta(\sigma, \rho) < 2^{-j+3} \sqrt{m} \).
\end{proof}

\begin{proof}[Proof of Proposition~\ref{propositionWspConvergenceStepFunction}]
Let \(\sigma, \rho \in K^{m}_{2^{-j}}\).
For \(x \in \sigma\) and \(y \in \rho\),
\[
\abs{E_j (v) (x) - E_j (v)(y)}
 \le \frac{1}{\abs{\sigma}\abs{\rho}} \int\limits_{\sigma}\int\limits_{\rho} \abs{v(\tilde x) - v(\tilde y)} \dif \tilde x \dif \tilde y.
\]
Thus, by Jensen's inequality,
\[
\abs{E_j (v) (x) - E_j (v)(y)}^p
 \le \frac{1}{\abs{\sigma}\abs{\rho}} \int\limits_{\sigma}\int\limits_{\rho} \abs{v(\tilde x) - v(\tilde y)}^p \dif \tilde x \dif \tilde y.
\]
We deduce that
\begin{equation}
\label{eqCubesAverageWsp}
\begin{split}
\int\limits_\sigma \int\limits_\rho \frac{\abs{E_j (v) (x) - E_j (v)(y)}^p}{\abs{x - y}^{m + sp}} \dif x \dif y
& \le \frac{C'}{\delta (\sigma, \rho)^{m + sp}} \int\limits_{\sigma}\int\limits_{\rho} \abs{v(\tilde x) - v(\tilde y)}^p \dif \tilde x \dif \tilde y\\
& \le C' \int\limits_{\sigma} \int\limits_{\rho}  \frac{\abs{v (\tilde x) - v (\tilde y)}^{p}}{\abs{x - y}^{m + sp}} \dif \tilde x \dif \tilde y.
\end{split}
\end{equation}
The desired estimate follows from \eqref{eqCubesAverageWsp} by summation over dyadic cubes in \(K^m_{2^{-j}}\).

To prove the convergence in \(W^{s, p}\) we write for every \(\lambda > 0\), 
\begin{multline*}
 [E_j (v) - v]^p_{W^{s, p}(Q^m)}\\
 \le 2^{p - 1} \iint\limits_{D_\lambda} \frac{\abs{E_j (v) (x) - E_j (v) (y)}^p + \abs{v (x) - v (y)}^p}{\abs{x -y}^{m + s p}} \dif x \dif y\\
 + \frac{2^{p} \abs{Q^m}}{\lambda^{m + s p}} \int\limits_{Q^m} \abs{E_j (v) - v}^p,
\end{multline*}
where
\[
 D_\lambda = \bigl\{ (x, y) \in Q^m \times Q^m : \abs{x - y}\le \lambda \bigr\}.
\]
By estimate \eqref{eqCubesAverageWsp},
\[
\begin{split}
\iint\limits_{D_\lambda} \frac{\abs{E_j (v) (x) - E_j (v) (y)}^p}{\abs{x -y}^{m + s p}} \dif x \dif y
 &\le \NewConstant \sum_{\substack{\sigma, \rho \in K^m_{2^{-j}}\\
 (\sigma \times \rho) \cap D_\lambda \ne \emptyset}} \int\limits_{\sigma} \int\limits_{\rho}  \frac{\abs{v (x) - v (y)}^{p}}{\abs{x - y}^{m + sp}} \dif x \dif y\\
 & \le \SameConstant \iint\limits_{D_{\lambda} + Q^{2m}_{2^{-j+1}}}  \frac{\abs{v (x) - v (y)}^{p}}{\abs{x - y}^{m + sp}} \dif x \dif y.
\end{split}
\]
Hence,
\begin{multline*}
 [E_j (v) - v]^p_{W^{s, p}(Q^m)}\\
 \le \Constant \iint\limits_{D_{\lambda} + Q^{2m}_{2^{-j+1}}} \frac{\abs{v (x) - v (y)}^p}{\abs{x -y}^{m + s p}} \dif x \dif y
 + \frac{2^{p} \abs{Q^m}}{\lambda^{m + s p}} \int\limits_{Q^m} \abs{E_j (v) - v}^p.
\end{multline*}
By Proposition~\ref{propositionEstimateELp}  the last integral tends to zero as \(j\) tends to infinity.
Thus,
\[
\limsup_{j \to \infty}{[E_j (v) - v]^p_{W^{s, p}(Q^m)}}
\le \SameConstant \iint\limits_{D_{\lambda}} \frac{\abs{v (x) - v (y)}^p}{\abs{x -y}^{m + s p}} \dif x \dif y.
\]
The conclusion follows by choosing \(\lambda > 0\) small enough.
\end{proof}

In the proof of Theorem~\ref{deuxiemetheorem} we need the following property from Differential topology:

\begin{proposition}
\label{propositionFiniteSet}
Let \(N^n\) be a connected manifold.
Then, for every finite subset \(A\) in \(N^n\), there exists an open neighborhood of \(A\) in \(N^n\) which is diffeomorphic to the Euclidean ball \(B^n\).
\end{proposition}

\begin{proof}
Let \(U \subset N^n\) be an open set which is diffeomorphic to the Euclidean ball \(B^n\).
There exists a diffeomorphism \(f : N^n \to N^n\) mapping \(A\) into \(U\) \cite{Hirsch}*{Lemma~5.2.6}; in dimension \(n \ge 2\) this follows from the multi-transitivity in the group of diffeomorphism of \(N^n\) \cite{Banyaga}*{Lemma~2.1.10}.
The set \(f^{-1}(U)\) is thus diffeomorphic to \(B^n\) and contains \(A\).
\end{proof}

\begin{proof}[Proof of Theorem~\ref{deuxiemetheorem} when \(sp < 1\)]
Let \(u\in W^{s,p}(Q^m ; N^n)\) and let \(\iota > 0\) be such that the nearest point projection \(\Pi\) into \(N^n\) is smooth on \(N^n + \overline{B}_\iota^\nu\).

Let \(b \in N^n\). 
For every \(j \in \N_*\), we define \(u_j : Q^m \to \R^\nu\) for \(x \in Q^m\) by
\[
u_j(x)=
\begin{cases}
E_j(u)(x)	& \text{if \(\dist{(E_j(u)(x) , N^n)} < \iota\),}\\
b	& \text{otherwise.}
\end{cases}
\]
Then, \((u_j)_{j \in \N_*}\) is a sequence of step functions with values into \(N^n+ B^{\nu}_{\iota}\).
By the triangle inequality,
\begin{equation}
\label{eqTriangleInequalityspless1}
\norm{u_j - u}_{W^{s, p}(Q^m)} 
 \le \norm{E_j(u) - u_j}_{W^{s, p}(Q^m)} + \norm{E_j(u) - u}_{W^{s, p}(Q^m)}.
\end{equation}
We need to estimate the first term in the right hand side of this inequality.
Since the range of \(E_j(u)\) is contained in a fixed bounded set --- for instance the convex hull of \(N^n\) ---, for every \(j\in \N_*\),
\[
\begin{split}
\norm{E_j(u) - u_j}_{L^p(Q^m)} 
& = \norm{E_j(u) - b}_{L^{p}(\{\dist{(E_j(u), N^n)} \geq \iota\})}\\
& \leq \NewConstant |\{x : \dist{(E_j(u)(x), N^n)} \geq \iota\}|^{\frac{1}{p}}.
\end{split}
\]
Since \(\abs{E_j(u)(x) - u(x)}\geq \iota\) on \(\{x : \dist{(E_j(u)(x), N^n)} \geq \iota\}\), we get
\[
\norm{E_j(u) - u_j}_{L^p(Q^m)} 
 \leq {\SameConstant} |\{x : \abs{E_j(u)(x) - u(x)} \geq \iota\}|^{\frac{1}{p}}.
\]
Thus, by the Chebyshev inequality,
\begin{equation}
\label{equationlpujej}
\norm{E_j(u) - u_j}_{L^p(Q^m)} 
 \leq \frac{\SameConstant}{\iota^\frac{1}{p}} \norm{E_j(u) - u}_{L^p(Q^m)}.
\end{equation}

We need a similar estimate for the Gagliardo seminorm \(W^{s, p}\):
\begin{Claim}
There exists \(C>0\) depending on \(s\), \(p\) and \(m\) such that for every \(j \in \N_*\)
\[
[E_j(u) - u_j]_{W^{s,p}(Q^m)} \le C \big( [E_j(u) - u]_{W^{s,p}(Q^m)} + [u]_{W^{s,p}(A_j)} \big),
\]
where \(A_j= \big\{ x\in Q^m : \dist{(E_j(u)(x) , N^n)} \ge \iota \big\}\).
\end{Claim}

\begin{proof}[Proof of the claim]
First note that 
\begin{multline*}
[E_j (u) - u_j]_{W^{s, p}(Q^m)}^p 
= 2 \sum_{\sigma \in \mathcal{A}} \sum_{\rho \in K^m_{2^{-j}} \setminus \mathcal{A}}
\;  \int\limits_{\sigma}\int\limits_{\rho} \frac{\abs{E_j (u) (x) - b}^p}{\abs{x - y}^{m + s p}}\dif x \dif y\\
+ \sum_{\sigma \in \mathcal{A}}\sum_{\rho \in \mathcal{A}}  \; \int\limits_{\sigma}\int\limits_{\rho} \frac{\abs{E_j (u) (x) - E_j (u) (y)}^p}{\abs{x - y}^{m + s p}}\dif x \dif y,
\end{multline*}
where 
\[
\mathcal{A} = \Bigl\{\sigma \in K^m_{2^{-j}} : \dist{(E_j(u)(x) , N^n)} \ge \iota \ \text{for}\ x \in \sigma \Bigr\} .
\] 
By \eqref{eqCubesAverageWsp}, we have
\[
\sum_{\sigma \in \mathcal{A}} \sum_{\rho \in \mathcal{A}}  \; \int\limits_{\sigma}\int\limits_{\rho} \frac{\abs{E_j (u) (x) - E_j (u) (y)}^p}{\abs{x - y}^{m + s p}}\dif x \dif y \leq \NewConstant [u]_{W^{s,p}(A_j)}^p.
\]
We now estimate the term
\[
I=\sum_{\sigma \in \mathcal{A}} \sum_{\rho \in K^m_{2^{-j}} \setminus \mathcal{A}}  \; 
  \int\limits_{\sigma}\int\limits_{\rho}  \frac{\abs{E_j (u) (x) - b}^p}{\abs{x - y}^{m + s p}}\dif x \dif y.
\]
Since the image of \(u\) is contained in \(N^n\) and \(N^n\) is bounded, there exists a constant \(\Constant > 0\) such that for every \(j \in \N_*\),
\[
\abs{E_j(u) - b} \le \SameConstant.
\] 
Since \(sp < 1\), by the Claim following Proposition~\ref{propositionWspConvergenceStepFunction},
\[
\begin{split}
I
 & \le  {\SameConstant^p} \sum_{\sigma \in \mathcal{A}} \sum_{\rho \in K^m_{2^{-j}} \setminus \mathcal{A}}
 \;   \int\limits_{\sigma}\int\limits_{\rho} \frac{1}{\abs{x - y}^{m + s p}} \dif x \dif y\\
 & \le \Constant \sum_{\sigma \in \mathcal{A}} \sum_{\rho \in K^m_{2^{-j}} \setminus \mathcal{A}}
  \frac{\abs{\sigma}\abs{\rho}}{\delta (\sigma, \rho)^{m + s p}}.
\end{split}
\]
For every \(\sigma \in \mathcal{A}\), 
\[
\int\limits_{\sigma} \abs{E_j (u) - u }^p \ge \iota^p \abs{\sigma}.
\]
Thus,
\[
I \le \frac{\SameConstant}{\iota^p} \sum_{\sigma \in \mathcal{A}} \sum_{\rho \in K^m_{2^{-j}} \setminus \mathcal{A}} \frac{\abs{\rho}}{\delta (\sigma, \rho)^{m + s p}}
 \int\limits_{\sigma} \abs{E_j (u) - u }^p.
\]
Since \(E_j(u) = \frac{1}{\abs{\rho}} \int_\rho u\) in \(\rho\), for \(x \in \sigma\) we have by the triangle inequality,
\[
\abs{E_j(u)(x) - u(x)} \le \frac{1}{\abs{\rho}} \int\limits_\rho \abs{E_j (u) (x) - u (x) - E_j (u) (y) + u(y)} \dif y.
\]
Thus, by Jensen's inequality, 
\[
\abs{E_j(u)(x) - u(x)}^p \le \frac{1}{\abs{\rho}} \int\limits_\rho \abs{E_j (u) (x) - u (x) - E_j (u) (y) + u(y)}^p \dif y.
\]
We deduce that 
\[
I 
\le \frac{\SameConstant}{\iota^p} \sum_{\sigma \in \mathcal{A}} \sum_{\rho \in K^m_{2^{-j}} \setminus \mathcal{A}} 
 \;  \int\limits_{\sigma} \int\limits_{\rho} \frac{\abs{E_j (u) (x) - u (x) -  E_j (u) (y) +u (y)}^p}{\abs{x - y}^{m + sp}} \dif y \dif x
\]
and the claim follows.
\end{proof}

By the triangle inequality \eqref{eqTriangleInequalityspless1}, by estimate \eqref{equationlpujej} and by the previous claim, we have for every \(j \in \N_*\),
\[
\norm{u_j - u}_{W^{s, p}(Q^m)} 
 \le {\Constant} \norm{E_j(u) - u}_{W^{s, p}(Q^m)} + C [u]_{W^{s,p}(A_j)}.
\]
Since \((E_j(u))_{j\in \N_*}\) converges to \(u\) in measure and \(u(x) \in N^n\) for a.e.~\(x \in Q^m\), the sequence \((|A_j|)_{j \in \N_*}\) converges to zero.
Since \(u \in W^{s,p}(Q^m)\), by the Dominated convergence theorem we get 
\[
\lim_{j\to +\infty}{[u]_{W^{s,p}(A_j)}}
=0.
\]
Applying Proposition~\ref{propositionWspConvergenceStepFunction}, we deduce that \((u_j)_{j \in \N_*}\) converges strongly to \(u\) in \(W^{s, p}(Q^m; \R^\nu)\).
Since \(u_j(Q^m) \subset N^n +\overline{B}_\iota^\nu\), the sequence \((\Pi \circ u_j)_{j \in \N_*}\) converges strongly to \(u\)  in \(W^{s, p}(Q^m; N^n)\).

To conclude the proof of Theorem~\ref{deuxiemetheorem}, we may then assume that \(u\) is a step function.
In this case, \(u(Q^m)\) is a finite set of points in \(N^n\). 
By Proposition~\ref{propositionFiniteSet}, there exists an open neighborhood \(U\) of \(u(Q^m)\) in \(N^n\) and a smooth diffeomorhism \(\Phi : \overline{U} \to \overline{B}^n\). Since the set \(\overline{B}^n\) is convex, there exists a sequence of smooth maps \((v_i)_{i \in \N}\) in \(C^\infty(\overline{Q}^m; \overline{B}^n)\) which converges strongly to \(\Phi \circ u\) in \(W^{s,p}(Q^m ; \overline{B}^n)\). 
Hence, the sequence \((\Phi^{-1} \circ v_i)_{i \in \N}\) converges strongly to \(u\) in \(W^{s,p}(Q^m ; N^n)\). 
This completes the proof of Theorem~\ref{deuxiemetheorem} for \(sp < 1\).  
\end{proof}


\subsection*{Acknowledgments}

The second (ACP) and third (JVS) authors were supported by the Fonds de la Recherche scientifique---FNRS.


\end{document}